\documentclass{article}

\usepackage{arxiv}

\usepackage[utf8]{inputenc} 
\usepackage[T1]{fontenc}    
\usepackage{url}            
\usepackage{booktabs}       
\usepackage{nicefrac}       
\usepackage{microtype}      
\usepackage{lipsum}
\usepackage{amsmath, amsthm, amscd, amsfonts, amssymb, graphicx, color}
\usepackage[backref,colorlinks=true]{hyperref}

\title{Classification of almost contact metric structures on three-dimensional non-unimodular Lie groups}

\author{
  Pejhman Vatandoost-Miandehi $^1$\\
  Department of Mathematics and Computer Science\\
  Amirkabir University of Technology\\
  Tehran, Iran \\
   \And
 Asadollah Razavi $^2$\\
  Department of Mathematics and Computer Science\\
  Amirkabir University of Technology\\
  Tehran, Iran \\
}

\newcommand{\g}{\mathfrak{g}}
\newcommand{\m}{\mathfrak{m}}
\newcommand{\fru}{\mathfrak{u}}
\newcommand{\cL}{\mathcal{L}}
\newcommand{\rspan}{\mathrm{span}}
\DeclareMathOperator{\ad}{\mathrm{ad}}

\DeclareMathOperator{\trace}{\mathrm{trace}}
\begin{document}
\maketitle
\newtheorem{theorem}{Theorem}[section]
\newtheorem{lemma}[theorem]{Lemma}
\newtheorem{proposition}[theorem]{Proposition}
\newtheorem{corollary}[theorem]{Corollary}
\newtheorem{question}[theorem]{Question}

\theoremstyle{definition}
\newtheorem{definition}[theorem]{Definition}
\newtheorem{algorithm}[theorem]{Algorithm}
\newtheorem{conclusion}[theorem]{Conclusion}
\newtheorem{problem}[theorem]{Problem}
\newtheorem{example}[theorem]{Example}

\theoremstyle{remark}
\newtheorem{remark}[theorem]{Remark}
\numberwithin{equation}{section}
\footnotetext[1]{E-mail:~\texttt{(pejhman.vatandoost@gmail.com)}; \texttt{(pejhman.vatandoost@iran.ir)}}
\footnotetext[2]{E-mail:~\texttt{(arazavi@iran.ir)}}
\begin{abstract}
In this paper, left-invariant almost contact metric structures on three-dimensional non-unimodular Lie groups are investigated. It is proved that for every Riemannian Lie group, there is one of these structures. In addition, left-invariant normal almost contact metric structures on three dimensional non-unimodular Lie groups are classified.
\end{abstract}

\keywords{Almost Contact Metric Structures, Homogenous Manifolds, Three-Dimensional non-Unimodular Riemannian Lie Groups}

\section{Introduction}
A manifold $M^{2n+1}$ has an almost contact structure  $(\varphi, \xi, \eta)$ if a $(1, 1)$-type tensor like $\varphi$ and a nowhere zero universal vector field $\xi$ and a form $\eta$ are provided on $M$ such that following conditions are satisfied:
\begin{equation}\label{e1}
\aligned
&\varphi(\xi)=1, \quad
\eta\circ \varphi=0, \quad
\eta(\xi)=1,\\
&\varphi^2=-Id + \eta \otimes \xi.
\endaligned	
\end{equation}
Now, if the manifold $M^{2n+1}$ with the almost contact structure $(\varphi, \xi, \eta)$  takes a Riemannian metric $g$, and the conditions 
\begin{equation}\label{e2}
\forall X, Y \in \chi(M), \quad
g(\varphi X, \varphi Y)= g(X, Y)- \eta(X) \eta(Y),
\end{equation}
are hold, then it can be said that $M^{2n+1}$ has an almost contact metric structure and $g$ is called a compatible metric or an almost contact metric.

Every almost contact structure takes a compatible metric. According to \eqref{e1} and \eqref{e2}, for every compatible metric, $\eta(X) = g(\xi, X)$  and also  $\ker \eta = \xi^\perp$, then  $g(\varphi X, Y) = g(X, \varphi Y)$, therefore $J=\varphi|_{ker\,\eta}$ is an compatible almost complex structure with restriction $g$ to $ker\,\eta$.

If $g$ is a compatible Riemannian metric with the almost contact structure $(\varphi, \xi, \eta)$  on the manifold $M$, $2$-form $\varphi$ that is defined as
\begin{equation}\label{e3}
\phi(X, Y) = g(X, \varphi Y),
\end{equation}
is called fundamental $2$-form. 

If for the almost contact structure $(\varphi, \xi, \eta)$  over the manifold $M$, there is a compatible Riemannian metric $g$ such that for every  $X, Y \in \chi(M)$, we have
\begin{equation}\label{e4}
d\eta (X, Y)=\phi (X, Y),
\end{equation}
then, $M$ is called contact Riemannian manifold and is represented with  $(M, \varphi, \xi, \eta)$. The almost contact metric manifold $(M, \varphi, \xi, \eta, g)$ is said to be homogeneous if a connected Lie subgroup $G$ from the group $M$ has a transitive relation with $M$ and $1$-form $\eta$ is invariant to $G$.

As stated in \cite{1}, to classify almost contact homogeneous 3D metric manifolds, it is sufficient to examine the left-invariant almost contact metric structures on 3D groups.

Suppose $\g$ is the next arbitrary $(2n + 1)$ Lie algebra. An almost contact metric structure on $\g$ is a quaternion  $(\varphi, \xi, \eta, g)$ in which $\eta$ is a $1$-form, $\varphi \in End(\g)$, and $g$ is a finite positive interior  on $\g$ such that
\begin{equation}\label{e5}
\aligned
\eta(\xi) =1, \quad \varphi^2=-Id + \eta \otimes \xi,\\
\forall X, Y \in \g, \quad g(\varphi X, \varphi Y) = g(X, Y) - \eta(X) \eta(Y).
\endaligned
\end{equation}
The two almost contact Lie algebras $(\g_1, \varphi_1, \xi_1, \eta_1, g_1)$  and  $(\g_2, \varphi_2, \xi_2, \eta_2, g_2)$ are isomorph if there is a linear mapping  $f: \g_1 \to \g_2$ called (contact) isomorphism, such that 
\begin{equation*}
f: (\g_1, g_1) \to (\g_2, g_2),
\end{equation*}
\begin{equation*}
f(\xi_1)=\xi_2,
\end{equation*}
\begin{equation*}
f \circ \varphi_1 = \varphi_2 \circ f,
\end{equation*}
\begin{equation*}
\eta_1 = \eta_2 \circ f,
\end{equation*}
Reference \cite{2} is suitable for the study of contact metric geometry. The contact geometry has been studied by many people. The references \cite{3,4,5,6,7,8} contain many examples of important results about contact geometry. In reference \cite{9}, Riemannian geometries on Lie groups equipped with left-invariant metrics are investigated.

In this paper, we will investigate the almost contact metric structures $(\varphi, \xi, \eta)$  over non-unimodular Lie groups that satisfy the condition  $\xi \in ker\,d\eta$.

\section{Almost contact metric structures on three dimensional non-unimodular Lie groups}
\begin{lemma}\label{lem 1}
	For every almost contact structure  $(\varphi, \xi, \eta)$, we have
	\begin{equation*}
	\xi \in \ker d\eta \Leftrightarrow \cL_\xi \eta=0,	
	\end{equation*}
	where $\cL$  is a Lie derivative. 
\end{lemma}

\begin{proof}
	According to \eqref{e1} we have
	\begin{equation*}
	\eta(\xi)=1.
	\end{equation*}
	Also for every vector field
	\begin{equation*}
	d\eta (\xi, X) = \xi (\eta(X)) - \eta([\xi, X]) 
	=
	(\cL_\xi \eta)(X).
	\end{equation*}
\end{proof}

\begin{proposition}\label{prop 2}
	Geodesic $\Gamma(t)$ of the manifold $M=K/H$ with the condition  $\Gamma(0)=0$ and
	$$\Gamma'(0) = X_\m \in \m,~~~(\m=T_0\left(\dfrac{K}{H}\right)),$$
	is homogeneous if and only if an $X_\eta \in \eta$  exists so that $X=X_m + X_\eta \in l$  in $g([X,Y]_\m, X_\m)=0$  is hold for all  $Y \in m$.
\end{proposition}
The vector $X \in l$  that is satisfied above condition is called geodesic vector.

\begin{proposition}\label{prop 3}
	Suppose $(M=K/H, \varphi, \xi, \eta, g)$  is a homogeneous almost contact metric manifold, then   $\xi \in \ker d\eta$ if and only if $\xi$ is a geodesic vector.
\end{proposition}

\begin{lemma}\label{lem 4}
	Suppose $(\varphi, \xi, \eta, g)$  is an almost contact metric structure with condition of  $\xi \in \ker d\eta$, then for every $X \in \ker \eta$  we have
	$$[\xi, X]\in \ker \eta.$$
\end{lemma}

\begin{proof}
	It is obvious.
\end{proof}

The three-dimensional Riemannian Lie groups are classified in \cite{9}. Suppose $G$ is a Lie group equipped with a left-invariant Riemannian metric $g$, as stated in \cite{9}, $G$ is unimodular if and only if the linear mapping
$$
L(x \times y)=[x, y], ~~~ x, y \in \g,
$$
is self-adjoint. So if $(G, g)$ is a three-dimensional unimodular Riemannian Lie group, then its Lie algebra $\g$ takes a basis $\{e_1, e_2, e_3\}$  such that
\begin{align*}
&  [e_1, e_2]= \lambda_3 e_3, \\
& [e_2, e_3]= \lambda_1 e_1, \\
& [e_3, e_1]= \lambda_2 e_2.
\end{align*}

Suppose $G$ is a non-unimodular three-dimensional connected Lie group. Then its Lie algebra has a basis  $\{e_1, e_2, e_3\}$ such that 	
\begin{equation}\label{e6}
[e_1, e_2]= \alpha e_2 + \beta e_3, \quad
[e_1, e_3]= \gamma e_2 + \delta e_3, \quad
[e_2, e_3]=0.
\end{equation}
Also matrix
\begin{equation*}
A= \left(\begin{matrix}
\alpha & \beta \\
\gamma & \delta
\end{matrix}\right),
\end{equation*}
satisfies conditions $\alpha \gamma + \beta \delta=0$  and  $\alpha + \delta=0$.

If $\g$ is a non-unimodular three-dimensional Lie algebra, its unimodular kernel i.e.,  
$$\fru = \{ x\in \g | \trace \ad (x)=0 \},$$ 
is $2$-dimensional and unimodular and commutative. Consider  
$$
e_1 \in \g, ~~~ \trace \ad (e_1)=2,
$$
since $\fru$ is commutative, a linear transformation
$$
L(\fru)=[e_1, \fru],
$$
from $\fru$ to itself with $\trace = 2$ is independent of  $e_1$ choice. Determinant $D= \dfrac{4(\alpha \delta- \beta \gamma)}{(\alpha+\delta)^2}$  of $L$ is a complete invariant isomorphism for this Lie algebra. By choosing  $e_2$, the vectors of  $L(e_2) =e_3$, $e_2$ are linearly independent and the conditions of $\trace(L)=2$ and $del(L)=D$ are satisfied by
\begin{equation*}
L(e_2) =e_3,~~~
L(e_3)=-D e_2 + 2 e_3.
\end{equation*}
Therefore, the bracket product operator is uniquely defined.

\begin{proposition}[Special case]\label{ec5}
	Suppose $\g$ is a Lie algebra with the property that the bracket product $[x, y]$ for every  $[x, y] \in \g$ is always equal to the linear combination of $x$ and $y$. Suppose  $dim~g \geq 2$, then 
	\begin{equation}\label{e7}
	[x,y]=l(x)y - l(y)x,
	\end{equation}
	where $l$ is a well-defined linear mapping from $\g$ to $\mathbb{R}$.
\end{proposition}

By choosing a definite positive metric, the shear curvatures are constant.
\begin{equation*}
K=-\|L\|^2 <0.
\end{equation*}
If  $(\varphi, \xi, \eta)$ is a left-invariant almost contact metric structure shadow on the non-unimodular Riemannian Lie group $(G, g)$, according to Proposition \ref{prop 3}, the vector field $\xi \in \g$  is a unit geodesic vector.

Now given the assumption of  $r \neq 0$ constants, there are p and q such that 
\begin{equation*}
\alpha=r+p, \quad
\delta=r-p, \quad
\beta=(r+p)q, \quad
\gamma=-(r-p)q.
\end{equation*}
Therefore, we will have the following possible modes:\\
\textbf{(A)} If  $p \neq \{ 0 , r, -r\}$, then the unit geodesic vector fields have only one of the following two states:
\begin{itemize}
	\item[A. 1)] $\pm e_1$  if  $\Delta := (\beta+\gamma)^2 - 4 \alpha \delta <0$;
	\item[A. 2)] $\pm e_1$  and $\cos \theta e_1 + \sin \theta e_3$  if $\Delta \geq 0$  where  $$\alpha \cos^2 \theta + (\beta+\gamma) \cos \theta \sin\theta + \delta \sin^2 \theta=0.$$
\end{itemize}

\noindent \textbf{(B)} If $p = r$, then the unit geodesic vector fields have only one of the following two states:
\begin{itemize}
	\item[B.1)]  $\pm e_1$ and $\pm e_2$  and $\pm \dfrac{1}{\sqrt{1+q^2}}(q e_2 - e_3)$  if  $q \neq 0$;
	\item[B.2)]  $\cos\theta e_1 + \sin\theta e_3$ if $q=0$.
\end{itemize}

\noindent\textbf{(C)}  If $p =-r$, then the unit geodesic vector fields have only one of the following two states:
\begin{itemize}
	\item[C.1)]  $\pm e_1$ and $\pm e_2$  and  $\pm \dfrac{1}{\sqrt{1+q^2}}(e_2 - qe_3)$ if  $q\neq0$;
	\item[C.2)]  $\cos\theta e_1 + \sin\theta e_2$ if $q=0$.
\end{itemize}

\noindent\textbf{(D)} If $p = 0$, then all unit vector fields in g are geodesic. (In fact $(G, g)$ have constant shear curvature).

\noindent\textbf{(E)} If condition \eqref{e7} is true, then all unit vector fields in $\g$ are geodesic.

Note that states \textbf{(B)} and \textbf{(C)} are isomorph. Simply replace $e_2$ with $e_2$ and $(\alpha, \beta)$ with $(\lambda, \delta)$.

Considering the above considerations, we classify states \textbf{(A)} to \textbf{(E)} at the isomorphism class as follows:

\noindent\textbf{(1)}	Suppose  $\xi=\pm e_1$.\\
Therefore, $\ker \eta= \xi^\perp = \rspan\{ e_1, e_2\}$ and since  $\{e_1, e_2\}$ is orthonormal, so it is an $\varphi$-basis, too. At the isomorphism class, we can assume $\xi=e_1$  and  $\varphi e_2=e_3$. Hence,  $(\varphi, \xi, \eta)$ at the isomorphism class is described as follows:
\begin{equation}\label{e8}
[\xi, e]=\alpha e + \varphi e \quad
[\xi, \varphi e] = \gamma e +  \delta \varphi e \quad
[e, \varphi e]=0,
\end{equation}
provided that $\alpha+\delta \neq 0$  and  $\alpha \gamma + \beta \delta =0$.\\

\noindent\textbf{(2)} Suppose  $\xi= \cos \theta e_2 + \sin \theta e_3$.\\
In this case, $\ker \eta = \xi^\perp =\{E_1 := e_1, \, E_2:=-\sin\theta e_2 + \cos \theta e_3 \}$  is an orthonormal basis and also an $\varphi$-basis. According to \eqref{e6}, we will have the following
\begin{align*}
& [\xi, E_1] = 
(\beta \cos^2\theta + (\delta-\alpha) \sin\theta\cos\theta -\gamma \sin^2\theta)E_2 
-
(\alpha \cos^2\theta + (\beta+\gamma) \sin\theta\cos\theta +\delta \sin^2\theta)\xi, \\
&
[\xi, E_2]=0, \\
&
[E_1, E_2]=
(\delta \cos^2\theta + (\beta-\gamma) \sin\theta\cos\theta +\alpha \sin^2\theta)E_2 
-
(\gamma \cos^2\theta + (\delta-\alpha) \sin\theta\cos\theta -\beta \sin^2\theta)\xi. 
\end{align*}
According to Lemma \ref{lem 4}, since  $E_1=\ker \eta$, then  $[\xi, E_1]\in \ker \eta$, that is
\begin{equation}\label{e9}
\alpha \cos^2\theta + (\beta+\gamma) \sin \theta \cos\theta +\delta \sin^2 \theta=0.
\end{equation}
This is precisely the condition of case \textbf{(A 2)}.

We assume
\begin{align*}
&  A:= \delta \cos^2 \theta - (\beta+\gamma) \sin\theta\cos\theta + \alpha \sin^2\theta, \\
& B:= -(\gamma \cos^2 \theta + (\delta-\alpha)\sin\theta\cos\theta -\beta \sin^2\theta),\\
& C:= \beta \cos^2\theta + (\delta - \alpha) \sin\theta\cos\theta - \gamma \sin^2 \theta.
\end{align*}
Given \eqref{e6} and \eqref{e9} we have
$$
A=A+0=\alpha+\delta \neq 0.
$$
As usual, at the isomorphism class, we assume  $e=E_1$ and  $\varphi e=E_2$. So  $(\varphi, \xi, \eta)$ is completely defined as follows:
\begin{equation}\label{e10}
[\xi, e]=C\varphi e \qquad
[\xi, \varphi e]=0, \qquad
[e, \varphi e]= A \varphi e + B \xi,
\end{equation}
provided that  $A\neq 0$. In this case, $\eta$ is contact form if and only if  $B \neq 0$.\\

\noindent\textbf{(3)} Suppose  $\xi = \pm \dfrac{1}{\sqrt{1+q^2}}(q e_2-e_3)$.\\
This happens when $p = r$ and  $q \neq 0$, that is,  $\alpha=2r\neq 0$,  $\beta=2rq \neq 0$,  $\gamma=\delta=0$.

An orthonormal basis for $ker~ \eta$ is as  $\{ e_1, \dfrac{1}{\sqrt{1+q^2}}(e_2+qe_3) \}$. So at the isomorphism class, we can assume
\begin{equation*}
\xi=\dfrac{1}{\sqrt{1+q^2}}(e_2- q e_3),~~~ e=e_1, ~~~
\varphi e = \dfrac{1}{\sqrt{1+q^2}}(e_2+qe_3).
\end{equation*}
Hence according to \eqref{e6} we have
\begin{equation}\label{e11}
[\xi, e]= -\beta \varphi e, \quad
[\xi, \varphi e]=0, \quad
[e, \varphi e]= \alpha \varphi e,
\end{equation}
if we put  $A:=\alpha$,  $B:=0$,  $C:=-\beta$, it becomes a special case of \eqref{e10}.\\

\noindent\textbf{(4)} Suppose  $\xi=\cos \theta e_1+ \sin \theta e_3$.\\
This occurs when $p = r$ and $q = 0$, that is,  $\alpha=2r\neq 0 =\beta= \gamma = \delta$. An orthonormal basis and also a $\varphi$-basis for $ker \eta$ is as  $\{e_2, -\sin \theta e_1+\cos \theta e_3\}$. At the isomorphism class, we can assume that  $e=e_2$, $\varphi e=-\sin \theta e_1+\cos \theta e_3$. Given \eqref{e6}, we have
$$
[\xi, e]=\alpha \cos \theta e,~~~ [\xi, \varphi e]=0,~~~ [e, \varphi e]=\alpha \sin \theta e.
$$
By taking 
\begin{equation*}
\begin{split}
& \bar{A}:= \alpha \cos \theta, \\
& \bar{B}:= \alpha \sin \theta,
\end{split}
\end{equation*}
the almost contact metric structure is completely defined as follows
\begin{equation*}
[\xi, e]=\bar{A}, \quad
[\xi, \varphi e]=0, \quad
[e, \varphi e]= \bar{B} e,
\end{equation*}
provided that  $\bar{A}^2+\bar{B}^2 \neq 0$. In this case, $\eta$ is never a contact form.\\

\noindent\textbf{(5)} Suppose $\xi$ is an arbitrary vector field and $p = 0$.\\
So we have  $\alpha=\delta=r\neq 0$ and  $\beta= -\gamma=rq$. There are real constants of $\theta$ and $\omega$ such that
\begin{equation*}
\xi  = \cos \theta e_1 + \sin \theta \cos \omega e_2 + \sin \theta \sin\omega e_3.
\end{equation*}
In addition, $ker~\eta=\{E_1, E_2 \}$  that 
\begin{equation*}
\begin{split}
&
E_1=-\sin\theta e_1 + \cos\theta \cos\omega e_2 +\cos\theta \sin \omega e_3, \\
&
E_2 = -\sin \omega +\cos \omega e_3.
\end{split}
\end{equation*}
So at the isomorphism class, we can assume that $e=E_1$  and  $\varphi e=E_1$. Given \eqref{e6}, we have
\begin{equation*}
[\xi, e]=\alpha \sin \theta \xi + \alpha \cos\theta e + \alpha q \varphi e.
\end{equation*}	
According to Lemma \ref{lem 4}, because $e \in ker~\eta$  then  $[\xi, e] \in ker~\eta$, and because  $\alpha \neq 0$, then  $\sin \theta=0$. So $\xi=e_1$  and  $\{e_2, e_3\}$ is a $\varphi$-orthonormal basis and also an $\varphi$-basis for $ker~\eta$. Then we have	
\begin{equation*}
[\xi, e]=\alpha,\quad
[\xi, \varphi e]=0, \quad
[e, \varphi e]=0,
\end{equation*}
provided  $\alpha \neq 0$. If  $\beta=\gamma=\delta$, then this is a special case of \eqref{e8}.\\

\noindent\textbf{(6)} Suppose $\xi$ is an arbitrary vector field and $\g$ is a specific case Lie algebra. So, there is a linear mapping like $l: \g \to \mathbb{R}$  such that  $$[x,y]=l(x)y -l(y)x,~~~\forall x,y\in \g.$$

Consider an arbitrary $\varphi$-basis like  $\{\xi, e, \varphi e \}$. Then we have:
\begin{align*}
&  [\xi, e]=l(x)e - l(e)\xi, \\
& [\xi, \varphi e] = l(\xi) \varphi e - l(\varphi e)\xi, \\
& [e, \varphi e]= l(e) \varphi e - l(\varphi e) e.
\end{align*}
According to Lemma \ref{lem 4}, because  $e, \varphi e \in ker~\eta$, then  $[\xi, e], [\xi, \varphi e] \in ker~\eta$, or  $l(e)=l(\varphi e)=0$. So, Lie algebra is completely represented by the real parameter  $\alpha:=l(\xi)$. $\alpha \neq 0$  because $\g$ is non-unimodular. Then we have
\begin{equation*}
[\xi, e]=\alpha e, \quad
[\xi, \varphi e]=\alpha \varphi e, \quad
[e, \varphi e]=0,
\end{equation*}
where  $\alpha \neq 0$. If  $\alpha=\delta$ and  $\beta=\gamma$, then this case is a special case of \eqref{e8}.

All the classifications we made are summarized as below:

\begin{theorem}\label{thm6}
	Suppose $\g$ is a three-dimensional non-unimodular Riemannian Lie algebra described by \eqref{e6} with respect to an appropriate orthonormal basis  $\{e_1, e_2, e_3\}$. Then at the isomorphism class, the following are the only possible modes for the left-invariant almost contact metric structures $\{\varphi, \xi, \eta\}$ where $\xi \in ker~\eta$\\
	(A)	If $\alpha$, $\beta$, $\gamma$ and $\delta$ satisfy \eqref{e6}
	\begin{equation}\label{e12}
	[\xi, e]=\alpha e + \beta \varphi e, \quad
	[\xi, \varphi e ]=\gamma e+\delta \varphi e, \quad
	[e, \varphi e]= 0.
	\end{equation}
	(B)	If $\alpha$, $\beta$, $\gamma$ and $\delta$ satisfy \eqref{e6} and \eqref{e9},
	\begin{equation}\label{e13}
	[\xi, e]=C \varphi e, \quad
	[\xi, \varphi e ]=0, \quad
	[e, \varphi e]= A \varphi e + B \xi.
	\end{equation}
	Given that  $A \neq 0$. \\
	\noindent(C)	If  $\alpha \neq 0=\beta=\gamma=\delta$
	\begin{equation*}
	[\xi, e]= \bar{A}^2 + \bar{B}^2,
	\end{equation*}
	given that $\bar{A}^2 + \bar{B}^2 \neq 0$. 
\end{theorem}

According to case (A) in Theorem \ref{thm6}, we obtain following corollary.
\begin{corollary}\label{cor7}
	Every three-dimensional non-unimodular Riemannian Lie algebra adopts a left-invariant almost contact metric structure of  $(\varphi, \xi, \eta, g)$ with condition  $\xi \in ker~d\eta$.
\end{corollary}

According to the classification of Theorem \ref{thm6}, we will get the following result:
\begin{corollary}\label{cor8}
	Suppose  $(\varphi, \xi, \eta, g)$ is a left-invariant almost contact metric structure on a three-dimensional non-unimodular Lie algebra g described as \eqref{e6} with an appropriate orthonormal basis  $\{e_1, e_2, e_3\}$. Then $\eta$ is a contact form if and only if \eqref{e9} exists and if  $B \neq 0$, then $(\varphi, \xi, \eta, g)$  is isomorph with state (b) of the theorm.
\end{corollary}

\section{Conclusion}
Every three dimensional non-unimodular Riemannian Lie algebra adopts a left--invariant almost contact metric structure $(\varphi, \xi, \eta, g)$   with the condition of  $\xi \in ker~d\eta$.
Suppose $(\varphi, \xi, \eta, g)$  is a left--invariant almost contact metric structure on a three-dimensional non-unimodular Lie algebra $g$ described as \eqref{e6} with an appropriate orthonormal basis  $\{e_1, e_2, e_3 \}$. Then $\eta$ is a contact form if and only if \eqref{e9} exists and if $B \neq 0$  then $(\varphi, \xi, \eta, g)$  is isomorph with state (b) in the theorem.

\end{document}